\documentclass[12pt,a4paper]{amsart}

\usepackage[english]{babel}
\usepackage{amsthm}

\usepackage[top=2cm,bottom=2cm,left=3cm,right=3cm,marginparwidth=1.75cm]{geometry}

\usepackage{amsmath}
\usepackage{graphicx}
\usepackage{ytableau}
\usepackage{booktabs}
\usepackage[colorlinks=true, allcolors=blue]{hyperref}
\usepackage{enumerate}

\usepackage[numbers,sort&compress]{natbib}

\usepackage{tikz,fullpage}
\usetikzlibrary{arrows,%
        petri,%
        topaths,%
        positioning,%
        shapes.geometric,
        decorations.pathmorphing}
\usepackage{tkz-berge}
\usepackage[position=bottom]{subfig}
\usepackage{pgf}

\usetikzlibrary{arrows,automata}
\usepackage{inputenc}

\newtheorem{thm}{Theorem}[section]

\newtheorem{lem}[thm]{Lemma}

\newtheorem{conj}[thm]{Conjecture}

\newtheorem{sublemma}{}[thm]

\theoremstyle{definition}
\newtheorem{defn}[thm]{Definition}

\date{\today}

\title{Binary normal networks without near reticulations can be reconstructed from their rooted triples}

\author{Andrew Francis}
\address{School of Mathematics and Statistics, University of New South Wales, Australia}
\email{a.francis@unsw.edu.au}

\author{Charles Semple}
\address{School of Mathematics and Statistics, University of Canterbury, New Zealand}
\email{charles.semple@canterbury.ac.nz}

\thanks{The first author is supported by the Australian Research Council, DP260102678. The second author was supported by the New Zealand Marsden Fund.}


\begin{document}

\begin{abstract}
Normal networks are an important class of phylogenetic networks that have compelling mathematical properties which align with intuition about inference from genetic data.  While tools enabling widespread use of phylogenetic networks in the biological literature are still under mathematical, statistical, and computational development, many such results are being assembled, and in particular for normal phylogenetic networks.
For instance, it has been shown that binary normal networks can be reconstructed from the sets of three- and four-leaf rooted phylogenetic trees that they display.  It is also known that one can reconstruct particular \emph{subclasses} of normal networks from just the displayed rooted triples.  This applies, for instance, to rooted binary phylogenetic trees and to binary level-$1$ normal networks. In this paper we address the question of how much of the class of binary normal networks can be reconstructed from just the rooted triples that they display. We find that all except those with substructures that we call ``near-sibling reticulations'' and ``near-stack reticulations'' can be reconstructed just from their rooted triples.  This goes some way to answering the natural question of how much information can be extracted from a set of displayed rooted triples, which are arguably the simplest substructure that one may hope for in a phylogenetic object.
\end{abstract}

\maketitle

\section{Introduction}

A significant challenge for the widespread adoption of phylogenetic networks for biological inference is the difficulty in reconstructing them from available data.  Such data typically includes alignments of genetic sequence, either at the whole genome level or, more commonly, at the gene level.  Such alignments can readily produce small substructures such as a rooted triple. Much of the difficulty lies in combining these substructures in a meaningful way. We can always reconstruct a phylogenetic network that infers, for example, each of the rooted triples in a given collection by simply choosing a network that infers {\em all} possible rooted triples. But this is neither meaningful nor informative. Thus we need to constrain the phylogenetic network that we reconstruct. One way to do this is to restrict the reconstructed network to a certain class. This is not unusual, as most current reconstruction methods in computational biology restrict to the class of phylogenetic trees. Furthermore, for the purposes of consistency, the substructures that we use for reconstruction need to have the property of being able to determine the phylogenetic networks of interest.

In the last twenty years, a wild assortment of classes of phylogenetic networks have been introduced and studied. Amongst these classes, normal and tree-child networks have arguably been the most prominent. Introduced by Willson~\cite{willson2008reconstruction, willson2010properties-of-n} and Cardona et al.~\cite{cardona2009comparison}, respectively, these two classes of phylogenetic networks are sufficiently complex to capture reticulate evolution but also sufficiently constrained to be meaningful. They also have the compelling biological property of ``visibility'', which allows the present to `see' all past speciation and reticulation events from extant species. However, of these two classes, normal networks are more adaptable to reconstruction~\cite{francis2025normal}. For example, unlike tree-child networks, normal networks are reconstructable from the small subtrees that they infer~\cite{linz2020caterpillars}. The purpose of this paper is to investigate this adaptability in the context of rooted triples, the first step in establishing a general method for taking a given collection of rooted triples and outputting a normal network that is representative of the initial data.

For a binary phylogenetic network $N$, let $R(N)$ denote the set of rooted triples displayed by $N$. It is well known that if $T$ is a rooted binary phylogenetic tree, then $R(T)$ determines $T$ (see, for example, \cite{semple2003phylogenetics}). In particular, we have the following theorem.

\begin{thm}
Let $T_1$ and $T_2$ be two rooted binary phylogenetic $X$-trees. Then $R(T_1)=R(T_2)$ if and only if $T_1\cong T_2$.
\label{2trees}
\end{thm}

While not difficult to prove, Theorem~\ref{2trees} underlies many so-called supertree methods. These methods take as input a collection (of not necessarily compatible) rooted phylogenetic trees on overlapping leaf sets and output a single rooted phylogenetic tree that ``best'' represents the input collection. For an excellent overview of supertree methods, see~\cite{bininda2004evolution}.  Furthermore, Theorem~\ref{2trees} justifies the use of rooted triples in the input collection, as the ``true'' tree (at least theoretically) is recoverable from the rooted triples it displays. More generally than trees, if $N$ is a binary level-$1$ normal network, then $R(N)$ determines $N$~\cite{Gambette2012}. However, not surprisingly, if $N$ is an arbitrary binary normal network, then $N$ cannot typically be determined by $R(N)$ alone (see, for example, the counterexample given by Lemma~\ref{l:near-sib.retics.same.triples}).  

While rooted triples alone are insufficient, it has recently been shown that displayed trees with at most four leaves are nevertheless sufficient for determining an arbitrary binary normal network. Let $Q(N)$ denote the set of $4$-leaf caterpillars, that is, rooted binary phylogenetic trees on $4$ leaves with exactly one cherry, displayed by a binary normal network $N$. The next theorem is established in~\cite[Theorem~2]{linz2020caterpillars}.
 
\begin{thm}
Let $N_1$ and $N_2$ be two binary normal networks on $X$. Then $R(N_1)=R(N_2)$ and $Q(N_1)=Q(N_2)$ if and only if $N_1\cong N_2$.
\label{quads}
\end{thm} 

Using an explicit example, it is shown in \cite{linz2020caterpillars} that the inclusion of $4$-leaf caterpillars is necessary in the statement of Theorem~\ref{quads}. This begs the question \emph{to what extent} are such caterpillars necessary? In this paper, we consider this question. More particularly, we investigate \emph{which binary normal networks are determined by the set of rooted triples they display}? 

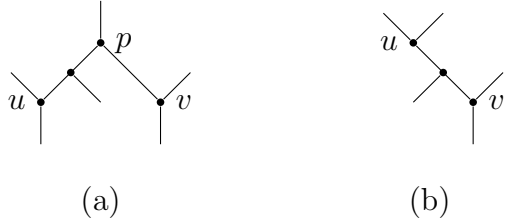
\begin{figure}[ht]
        
\begin{tikzpicture}[ 
   sdot/.style={circle,fill,radius=1pt,inner sep=1pt}] 

\node[sdot,label={left:{$u$}}] (r1) {}; 
\node[above right=1cm of r1,sdot,label={right:{$p$}}] (parent) {};
\node[below right=1cm of parent,sdot,label={right:{$v$}}] (r2) {};
\node[above=5mm of parent] (gp) {};
\node[above right=5mm of r2] (r2p) {};
\node[above left=5mm of r1] (r1p) {};
\node[below=5mm of r1] (r1c) {};
\node[below=5mm of r2] (r2c) {};

\draw (r1) to node[pos=0.5,sdot] (t1) {} (parent);
\node[below right=5mm of t1] (t1c) {};
\draw (gp)--(parent)--(r2)--(r2c) (r1p)--(r1)--(r1c) (r2p)--(r2) (t1)--(t1c);

\node[below=17mm of parent] (a) {(a)};

\node[sdot,right=4cm of r2,label={right:{$v$}}] (r2) {};
\node[sdot,above left=1cm of r2,label={left:{$u$}}] (r1) {};
\node[above left=5mm of r1] (r1p1) {};
\node[above right=5mm of r1] (r1p2) {};
\node[above right=5mm of r2] (r2p) {};
\node[below=5mm of r2] (r2c) {};

\draw (r1p1)--(r1)--(r1p2);
\draw (r1) to node[pos=0.5,sdot] (t1) {} (r2);
\node[below left=5mm of t1] (t1c) {};
\draw (t1)--(t1c);
\draw (r2p)--(r2)--(r2c);

\node[right=35mm of a] () {(b)};

\end{tikzpicture}
\caption{The reticulate vertices $u$ and $v$ in each figure are near-reticulations. In (a) they are near-sibling reticulations, where the vertex $p$ is a parent of $v$ and a grandparent of $u$, and in (b) they are near-stack reticulations. 
}
\label{f:near-sibs}
\end{figure}

Let $u$ and $v$ be reticulations of a binary phylogenetic network $N$. If there is a tree vertex in $N$ such that one of its children is a tree vertex, and this tree vertex is a parent of $u$, and the other child is $v$, then we say that $u$ and $v$ are {\em near-sibling reticulations}. Furthermore, if the unique child of $u$ is a tree vertex and this tree vertex is a parent of $v$, then we say that $u$ and $v$ are {\em near-stack reticulations}. These concepts are illustrated in Figure~\ref{f:near-sibs}. If $N$ has no near-sibling and no near-stack reticulations, we say $N$ has no {\em near reticulations}. Note that binary level-$1$ normal networks have no near reticulations. The main result of this paper is the following theorem.

\begin{thm}
Let $N_1$ and $N_2$ be two binary normal networks on $X$ with no near reticulations. Then $R(N_1)=R(N_2)$ if and only if $N_1\cong N_2$.
\label{main}
\end{thm}

As we show in Section~\ref{properties}, the condition in the statement of Theorem~\ref{main} that $N_1$ and $N_2$ have no near-sibling reticulations is a necessary condition for the theorem to hold. In particular, if $N$ is a normal network on $X$ with a pair of near-sibling reticulations, then, {unless the reticulations have a certain structural property}, there is a normal network $N'$ on $X$ not isomorphic to $N$ such that $R(N)=R(N')$. However, it remains open as to whether excluding near-stack reticulations is also necessary in the statement. {We discuss this further in the last section. For now,} we make the following conjecture.

\begin{conj}\label{c:near.sib.enough}
Let $N_1$ and $N_2$ be two binary normal networks on $X$ with no near-sibling reticulations. Then $R(N_1)=R(N_2)$ if and only if $N_1\cong N_2$.
\end{conj}

The paper is organised as follows. The next section contains some necessary preliminaries including formal definitions. Section~\ref{properties} reviews and establishes some properties of normal networks. In particular, we show that having no near-sibling reticulations in the statement of Theorem~\ref{main} is a necessary condition. The proof of Theorem~\ref{main}, which follows the model of the proof of Theorem~\ref{quads}, relies on being able to recognise cherries and reticulated cherries of a binary normal network $N$ with no near reticulations using only the set $R(N)$ of rooted triples displayed by $N$. This recognition is established in Section~\ref{recognition}, while the proof of Theorem~\ref{main} is given in Section~\ref{proof}. We end the paper with a discussion in Section~\ref{discussion}.

\section{Preliminaries}

Throughout the paper $X$ denotes a non-empty finite set.

\noindent {\bf Phylogenetic networks and trees.} A \emph{binary phylogenetic network $N$ on $X$} is an acyclic directed graph with (i) a single vertex of in-degree~$0$, which has out-degree~$2$ and is called the \emph{root}, (ii) $|X|$ vertices of in-degree~$1$ and out-degree $0$ labelled bijectively by the elements in $X$, and (iii) all other vertices having either in-degree~$1$ and out-degree~$2$, called {\em tree vertices}, or in-degree~$2$ and out-degree~$1$, called {\em reticulations}. The set $X$ is the {\em leaf set} of $N$, and the elements in $X$ are called {\em leaves}. Furthermore, the arcs directed into a reticulation are called {\em reticulation arcs}. For technical reasons, if $|X|=1$, then the acyclic directed graph consisting of a single vertex labelled by the element in $X$ is a binary phylogenetic network on $X$. For a vertex $u$ of $N$, the {\em cluster set} of $u$, denoted by $C_u$, is the subset of $X$ consisting of the leaves of $N$ that are descendants of $u$ (those leaves for which there is a directed path from $u$). A {\em rooted binary phylogenetic $X$-tree} is a phylogenetic network on $X$ with no reticulations. A {\em rooted triple} is a binary phylogenetic tree $T$ on three leaves. If the leaf set of a rooted triple $T$ is $\{x, y, z\}$ and $z$ is adjacent to the root, then we denote $T$ by $xy|z$ or, equivalently, $yx|z$. Since all phylogenetic networks and phylogenetic trees considered in this paper are rooted and binary, we refer to a rooted binary phylogenetic network and a rooted binary phylogenetic tree as a phylogenetic network and phylogenetic tree, respectively.

Let $N_1$ and $N_2$ be two phylogenetic networks on $X$ with vertex sets $V_1$ and $V_2$, respectively. Then $N_1$ is {\em isomorphic} to $N_2$ if there is bijective map $\varphi:V_1\rightarrow V_2$ such that $\varphi(x)=x$ for all $x\in X$, and $(u, v)$ is an arc of $N_1$ if and only if $(\varphi(u), \varphi(v))$ is an arc of~$N_2$.

\noindent {\bf Tree-child and normal networks.} A phylogenetic network $N$ is \emph{tree-child}  if every non-leaf vertex has at least one child that is a tree vertex or a leaf. Equivalently, each vertex $v$ of $N$ has a \emph{tree path} to a leaf, that is, a path from $v$ to a leaf whose non-terminal vertices are tree vertices. A reticulation arc $(u, v)$ is a {\em shortcut} if there is a (directed) path from $u$ to $v$ that does not traverse $(u, v)$. A \emph{normal network} is a tree-child network without shortcuts. Tree-child and normal networks were introduced in \cite{cardona2009comparison} and \cite{willson2008reconstruction,willson2010properties-of-n}, respectively.

Tree-child networks and normal networks have the important property that every vertex is ``visible''. A vertex $u$ of a phylogenetic network on $X$ is {\em visible} if there is a leaf $\ell$ such that every path from the root of $N$ to $\ell$ traverses $u$, in which case $\ell$ {\em verifies the visibility} of $u$. The {\em visibility set} of $u$, denote by $V_u$, is the subset of $X$ consisting of those leaves that verify the visibility of $u$. Observe that $V_u\subseteq C_u$ and that $C_u-V_u$ is not necessarily empty. It turns out that visibility characterises the class of tree-child networks. The next lemma is well known and gives this characterisation as well as the other characterisation that is central to this paper. Its proof is an almost immediate consequence of the definition. A pair of reticulations in a phylogenetic network are \emph{sibling reticulations} if they share a parent and \emph{stack reticulations} if one of the reticulations is a parent of the other reticulation. 

\begin{lem}
Let $N$ be a phylogenetic network. Then the following statements are equivalent:
\begin{enumerate}[{\rm (i)}]
\item $N$ is tree-child.

\item Every vertex of $N$ is visible.

\item $N$ has no sibling and no stack reticulations.
\end{enumerate}
\label{tree-child}
\end{lem}

Let $u$ and $v$ be reticulations of a phylogenetic network $N$. Repeating the definitions given in the introduction, if there is a tree vertex of $N$ whose two children are $v$ and a tree vertex parent of $u$, then $u$ and $v$ are {\em near-sibling reticulations}. Furthermore, if the child of $u$ is a tree vertex and a parent of $v$, then $u$ and $v$ are {\em near-stack reticulations}. A phylogenetic network with no near-sibling and no near-stack reticulations is said to have no {\em near reticulations}.

Let $N$ be a phylogenetic network on $X$, and let $\{a, b\}\subseteq X$. We say that $\{a, b\}$ is a {\em cherry} of $N$ if $a$ and $b$ have the same parent. Furthermore, $\{a, b\}$ is a {\em reticulated cherry} if the parent $p_b$ of $b$ is a reticulation and the parent $p_a$ of $a$ is a parent of $p_b$, in which case $b$ is the {\em reticulation leaf} of the reticulated cherry. Observe that $p_a$ is necessarily a tree vertex. The next lemma shows that every tree-child network contains either a cherry or a reticulated cherry~\cite[Lemma~4.1]{bordewich2016determining}. 

\begin{lem}
Let $N$ be a tree-child network on $X$. If $|X|\ge 2$, then $N$ contains either a cherry or a reticulated cherry.
\label{cherries}
\end{lem}

\noindent {\bf Displaying and embedding phylogenetic trees.} Let $N$ be a phylogenetic network on $X$, and let $T$ be a phylogenetic $X'$-tree, where $X'\subseteq X$. We say $N$ {\em displays} $T$ if there is a subdigraph $N'$ of $N$ such that, up to suppressing vertices of in-degree one and out-degree one, $T$ is isomorphic to $N'$, in which case the arc set of $N'$ is an {\em embedding} of $T$ in $N$.

In this paper, it is the embedding of rooted triples that is of interest. Let $N$ be a phylogenetic network, and suppose that $E_{xyz}$ is an embedding of the rooted triple $xy|z$ in $N$. If $e$ is an arc of $N$, we say that $E_{xyz}$ {\em uses $e$} if $e\in E_{xyz}$. Similarly, if $u$ is a vertex of $N$, we say that $E_{xyz}$ {\em uses $u$} if $u$ is an end vertex of an arc in $E_{xyz}$. Extending this terminology, let $f$ be an arc of $xy|z$. If the arcs of $E_{xyz}$ corresponding to $f$ use an arc $e$ of $N$, we say {\em $f$ uses $e$ in $E_{xyz}$}. Similarly, if $v$ is a vertex of $N$ and the arcs of $E_{xyz}$ corresponding to $f$ use $v$, we say {\em $f$ uses $v$ in $E_{xyz}$}.

\section{Properties of Normal Networks}
\label{properties}

We begin this section by showing that the condition of having no near-sibling reticulations is a necessary condition in the statement of Theorem~\ref{main}. Let $N$ be a normal network on $X$, and suppose that $u$ and $v$ are near-sibling reticulations of $N$, where a parent $t$ of $v$ is a grandparent of $u$, as shown in the left of Figure~\ref{f:near-siblings.same.triples}. Let $s$ denote the parent of $u$ that is a child of $t$, and let $w$ denote the child of $s$ that is not $u$. Since $N$ is normal, $w$ is either a tree vertex or a leaf. Let $N'$ be the phylogenetic network on $X$ obtained from $N$ by deleting $(s, w)$, suppressing $s$, subdividing the edge $(t, v)$ with a new vertex $s'$, and adjoining $w$ to this new vertex with a new edge. This operation on $N$ is illustrated in Figure~\ref{f:near-siblings.same.triples} and we say that $N'$ has been obtained from $N$ by a {\em nearest neighbour interchange relative to $\{u, v\}$}. {Certainly, $N$ is not isomorphic to $N'$ and, provided $(t, u)$ is not a shortcut in $N'$, it is easily checked that, as $N$ is normal, $N'$ is also normal. Note that if $(t, u)$ is a shortcut in $N'$, then $v$ is an ancestor of $u$, that is, $u$ and $v$ are {\em comparable} in $N$ (and $N'$).} The next lemma shows that we cannot distinguish $N$ and $N'$ based only on their sets of rooted triples. 
The case $(t, u)$ is a shortcut is discussed in Section~\ref{discussion}.

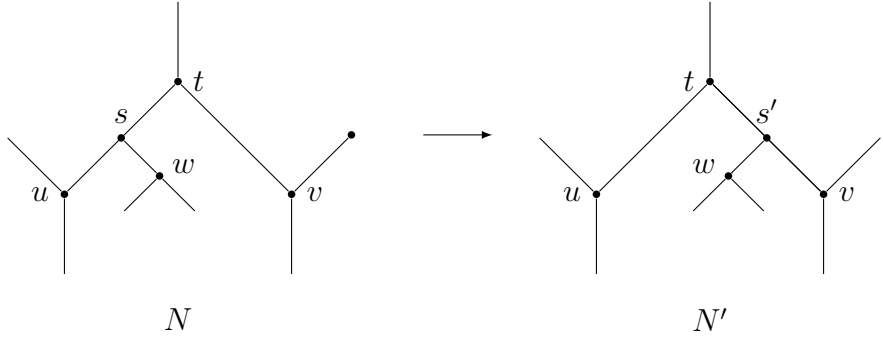
\begin{figure}[ht]
\begin{tikzpicture}[
   sdot/.style={circle,fill,radius=1pt,inner sep=1pt}]

\node[sdot,label={left:{$u$}}] (r1) {}; 
\node[above right=2cm of r1,sdot,label={right:{$t$}}] (parent) {}; 
\node[below right=2cm of parent,sdot,label={right:{$v$}}] (r2) {}; 
\node[above=10mm of parent] (gp) {}; 
\node[sdot,above right=10mm of r2] (r2p) {}; 
\node[above left= 10mm of r1] (r1p) {}; 
\node[below=10mm of r1] (r1c) {}; 
\node[below=10mm of r2] (r2c) {}; 

\draw[] (r1) to node[pos=0.5,sdot,label={above:{$s$}}] (t1) {} (parent);
\node[sdot,below right=6mm of t1] (t1c) {}; 
\node[below left=6mm of t1c] (t1cl) {};
\node[below right=6mm of t1c] (t1cr) {};

\draw (gp)--(parent) (r2)--(r2c) (r1p)--(r1)--(r1c);
\draw (r2p) to (r2);
\draw (parent) to (r2);
\draw[] (t1) to node[pos=.8,label={right:{$w$}}] () {} (t1c.north west);
\draw (t1cl)--(t1c)--(t1cr);

\node[right=6mm of r2p] (larrow) {};
\node[right=1.8cm of r2p] (rarrow) {};
\draw[>=latex, ->] (larrow)--(rarrow);

\node[below=28mm of parent] () {$N$};

\begin{scope}[xshift=7cm]
\node[sdot,label={left:{$u$}}] (r1) {}; 
\node[above right=2cm of r1,sdot,label={left:{$t$}}] (parent) {};
\node[below right=2cm of parent,sdot,label={right:{$v$}}] (r2) {};
\node[above=10mm of parent] (gp) {};
\node[above right=10mm of r2] (r2p) {};
\node[above left= 10mm of r1] (r1p) {};
\node[below=10mm of r1] (r1c) {};
\node[below=10mm of r2] (r2c) {};

\draw[] (r1) to (parent);
\draw (r2) to node[pos=0.5,sdot,label={above:{$s'$}}] (t1) {} (parent);
\node[sdot,below left=6mm of t1] (t1c) {};
\node[below left=6mm of t1c] (t1cl) {};
\node[below right=6mm of t1c] (t1cr) {};

\draw (gp)--(parent)--(r2)--(r2c) (r1p)--(r1)--(r1c) (r2p)--(r2);
\draw[] (t1) to node[pos=.8,label={left:{$w$}}] () {} (t1c.north east);
\draw (t1c.north east) to (t1c);
\draw (t1cl)--(t1c)--(t1cr);

\node[below=28mm of parent] () {$N'$};
\end{scope}
\end{tikzpicture}
\caption{Normal networks $N$ and $N'$, where $N'$ has been obtained from $N$ by a nearest neighbour interchange relative to $\{u, v\}$.}
\label{f:near-siblings.same.triples}
\end{figure}

\begin{lem}
\label{l:near-sib.retics.same.triples}
Let $N$ be a normal network with near-sibling reticulations $u$ and $v$, and let $N'$ be the {phylogenetic} network obtained from $N$ by a nearest-neighbour interchange relative to $\{u, v\}$. Then $R(N)=R(N')$. {In particular, if $u$ and $v$ are non-comparable, then $N'$ is normal and $R(N)=R(N')$.}
\end{lem}

\begin{proof}
Without loss of generality, we may assume that a parent of $v$ is a grandparent of~$u$. In $N$, let $s$ and $t$ denote parents of $u$ and $v$, respectively, such that $t$ is a parent of $s$. Let $w$ denote the tree vertex or leaf child of $s$, and let $e=(s, w)$. 

In $N'$, let $t$ and $s'$ denote the parents of $u$ and $v$, respectively, such that $t$ is a parent of $s'$. Let $xy|z\in R(N)$ and let $E_{xyz}$ be an embedding of $xy|z$ in $N$. Using $E_{xyz}$, we will show that there is an embedding $E'_{xyz}$ of $xy|z$ in $N'$. If $E_{xyz}$ avoids the arc $e$, then it is clear that $xy|z$ is a rooted triple of $N'$, so suppose that $E_{xyz}$ uses $e$.

Label the arcs of the rooted triple $xy|z$ by $f_x$, $f_y$, $f_{xy}$, and $f_z$, where $f_x$, $f_y$, and $f_z$ are the pendant arcs incident with $x$, $y$, and $z$, respectively (see Figure~\ref{f:edge.labels}).
First assume that $f_z$ uses $e$. Say $f_z$ uses $t$. If $f_{xy}$ avoids $t$, then set
$$E'_{xyz} = \big(E_{xyz}-\big\{(t, s),(s, w)\big\}\big)\cup \big\{(t, s'),(s', w)\big\},$$
while if $f_{xy}$ uses $t$, then set
$$E'_{xyz} = \big(E_{xyz}-\big\{(t, s), (s, w), (t, v)\big\}\big)\cup \big\{(s', w), (s', v)\big\}.$$
Otherwise, if $f_z$ avoids $t$, then set
$$E'_{xyz} = \big(E_{xyz}-\big\{(s, w), (s, u)\big\}\big)\cup \big\{(t, u), (t, s'), (s', w)\big\}.$$
Since $E_{xyz}$ is an embedding of $xy|z$ in $N$, it is easily checked that $E'_{xyz}$ is an embedding of $xy|z$ in $N'$. Taking a similar approach, if $f_{xy}$ uses $e$, it is straightforward to modify $E_{xyz}$ to get an embedding of $xy|z$ in $N'$.
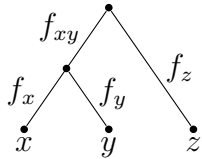
\begin{figure}[ht]
\begin{tikzpicture}[sdot/.style={circle,fill,radius=1pt,inner sep=1pt},label distance=-1mm]
\node[sdot] (root) {};
\node[sdot,below=15mm of root,label={below:{$y$}}] (y) {};
\node[sdot,right=10mm of y,label={below:{$z$}}] (z) {};
\node[sdot,left=10mm of y,label={below:{$x$}}] (x) {};
\draw (root) to 
   node[pos=.15,label={left:{$f_{xy}$}}] () {}
   node[pos=.5,sdot] (xy) {} 
   node[pos=.7,label={left:{$f_{x}$}}] () {}
   (x);
\draw (xy) to
   node[pos=.4,label={right:{$f_{y}$}}] () {}
   (y);
\draw (root) to 
   node[pos=.5,label={right:{$f_{z}$}}] () {}
   (z);
\end{tikzpicture}
\caption{Arc labels of the rooted triple $xy|z$ used in the proof of Lemma~\ref{l:near-sib.retics.same.triples}.}
\label{f:edge.labels}   
\end{figure}

Now assume that $f_x$ uses $e$. If $f_x$ uses $t$, then a similar modification to that in the previous paragraph gives an embedding of $xy|z$ in $N'$. Say $f_x$ avoids $t$. Then $f_{xy}$ uses $(t,s)$. If $f_z$ avoids $(t, v)$, then again a modification of $E_{xyz}$ similar to that in the previous paragraph gives an embedding of $xy|z$ in $N'$. Say $f_z$ uses $(t, v)$. Let $P_{t}$ be a path from the root of $N'$ to $t$ and let $P_v$ be a path from the root of $N'$ to $v$ using the reticulation arc directed into $v$ that is not $(s', v)$.

If $r$ denotes the last vertex common to $P_{t}$ and $P_v$, let $E_{t}$ and $E_v$ denote the arcs of $P_{t}$ and $P_v$ from $r$ to $t$ and from $r$ to $v$, respectively. Set
$$E'_{xyz} = \big(E_{xyz}-\big\{(s, u), (s, w), (t, s), (t, v)\big\}\big)\cup \big\{t, u), (s', w), (t, s')\big\}\cup E_{t}\cup E_v.$$
Then $E'_{xyz}$ is an embedding of $xy|z$ in $N'$ unless $v$ is a descendant of either $u$ or $w$ in $N$. But then $(t, v)$ is a shortcut in $N$, contradicting that $N$ is normal. By symmetry, if $f_y$ uses $e$, then we can modify $E_{xyz}$ to get an embedding of $xy|z$ in $N'$. Thus $xy|z\in R(N')$ and so $R(N)\subseteq R(N')$. The same argument but with the roles of $N$ and $N'$ interchanged gives $R(N')\subseteq R(N)$. This completes the proof of the lemma.
\end{proof}

The next lemma relates a reticulated cherry of a normal network with the rooted triples it displays.

\begin{lem}
\label{l:Vgb.properties}
Let $N$ be a normal network on $X$. Let $\{a, b\}$ be a reticulated cherry of $N$ with reticulation leaf $b$, and let $g_b$ be the grandparent of $b$ that is not the parent of $a$. Then each of the following hold:
\begin{enumerate}[{\rm (i)}]
\item For all $x\in X-\{a, b\}$, we have $ab|x\in R(N)$.

\item For all $c\in V_{g_b}$ and $x\in X-(V_{g_b}\cup \{b\})$, we have $bc|x\in R(N)$ and $ac|b\not\in R(N)$.

\item For all distinct $c, c'\in V_{g_b}$, we have $bc|c'\not\in R(N)$.

\item For all $c\in V_{g_b}$ and $x\in X-(V_{g_b}\cup \{a, b\})$, we have $cx|a\in R(N)$ if and only if $bx|a\in R(N)$.

\item If $bx|y\in R(N)$ and $ax|y\not\in R(N)$, where $x, y\in X-(V_{g_b}\cup \{a, b\})$, then $cx|y\in R(N)$ for all $c\in V_{g_b}$.

\item If $N$ is normal with no near reticulations, $c\in V_{g_b}$, and $ac|x\in R(N)$ for all $x\in X-(V_{g_b}\cup \{a, b\})$, then $ax|c, ax|b\not\in R(N)$. 
\end{enumerate}
\end{lem}

\begin{proof} 
Items (i)--(iii) are established in~\cite{linz2020caterpillars}. For parts (iv)--(vi), Figure~\ref{f:proof.of.5} can be helpful in visualising the arguments below.  

Let $p_a$ and $p_b$ denote the parents of $a$ and $b$ in $N$. To prove (iv), let $c\in V_{g_b}$ and $x\in X-(V_{g_b}\cup \{a, b\})$. If $cx|a\in R(N)$, then, as $c$ is in the visibility set of $g_b$, an embedding $E_{cxa}$ of $cx|a$ in $N$ uses $g_b$ and avoids $p_b$. Let $f_c$ denote the arc of $cx|a$ incident with $c$. If $f_c$ avoids $g_b$ in $E_{cxa}$, in which case $x\in C_{g_b}$, set $E_c$ to be the subset of arcs of $E_{cxa}$ used by $f_c$. On the other hand, if $f_c$ uses $g_b$ in $E_{xca}$, set $E_c$ to be the subset of arcs of $E_{cxa}$ that induce a path in $N$ from $g_b$ to $c$. In both instances, it is easily seen that
$$(E_{cxa}-E_c)\cup \{(g_b, p_b), (p_b, b)\}$$
is an embedding of $bx|a$ in $N$. The converse is simpler and omitted, completing the proof of (iv).

For the proof of (v), let $x, y\in X-(V_{g_b}\cup \{a, b\})$, and suppose that $bx|y\in R(N)$ but $ax|y\not\in R(N)$. Let $E_{bxy}$ be an embedding of $bx|y$ in $N$. Since $ax|y\not\in R(N)$, it follows that $E_{bxy}$ uses $g_b$ and the arc directed into $g_b$. Therefore, if $c\in V_{g_b}$ and $P_c$ is a path from $g_b$ to $c$, then
$$(E_{bxy}-\{(g_b, p_b), (p_b, b)\})\cup E(P_c)$$
is an embedding of $cx|y$ in $N$ unless $P_c$ intersects the path in $E_{bxy}$ corresponding to the pendant edge of $bx|y$ incident with $y$. But $y\not\in V_{g_b}$, so this is not possible; otherwise, there is a path from the root of $N$ to $c$ avoiding $g_b$. This completes the proof of (v).

To prove (vi), suppose that $N$ has no near reticulations, $c\in V_{g_b}$, and $ac|x\in R(N)$ for all $x\in X-(V_{g_b}\cup \{a, b\})$. Let $g_a$ denote the parent of $p_a$, as shown in Figure~\ref{f:proof.of.5}. If $g_a$ is not an ancestor of $g_b$ (Figure~\ref{f:proof.of.5}(b)), then $g_a$ is a tree vertex; otherwise, $g_a$ and $p_b$ are near-stack reticulations. If $g_a$ is the parent of a reticulation, then this reticulation and $p_b$ are near-sibling reticulations, a contradiction. It now follows that there is a tree path from $g_a$ to a leaf $\ell\not\in V_{g_b}\cup \{a, b\}$. But then, if $c\in V_{g_b}$, the rooted triple $ac|\ell\not\in R(N)$, a contradiction.

Now assume that $g_a$ is an ancestor of $g_b$ (Figure~\ref{f:proof.of.5}(c)).  A similar argument to that in the last paragraph shows that the parent of $g_b$ is an ancestor of $a$.  Since $N$ is acyclic, this implies that $g_a$ is the parent of $g_b$. It now follows that $ax|c, ax|b\not\in R(N)$ for all $x\in X-(V_{g_b}\cup \{a, b\})$. This completes the proof of (vi) and the lemma.
\end{proof}

\begin{figure}[ht]
\scalebox{.8}{
\begin{tikzpicture}[sdot/.style={circle,fill,radius=1pt,inner sep=1pt}]

\node[sdot,label={below:{$b$}}] (b) {}; 
\node[sdot,above=10mm of b,label={right:{$p_b$}}] (parent) {}; 
\node[sdot,above right=15mm of parent,label={right:{$g_b$}}] (p2) {}; 
\node[sdot,above left=15mm of parent,label={right:{$p_a$}}] (p1) {}; 
\node[above=6mm of p2] (p2p) {}; 
\node[above=6mm of p1] (p1p) {}; 
\node[below right=10mm of p2,draw,rectangle,label={below:{$C_{g_b}$}}] (ell2) {
   \begin{tikzpicture}
      \node[below right=10mm of p2,draw,dashed,ellipse] (VGb) {$V_{g_b}$};
   \end{tikzpicture}
}; 
\node[sdot,below left=15mm of p1,label={below:{$a$}}] (ell1) {}; %
\node[below left=5mm of p1p] (p1pc) {};
\draw (p2p)--(p2)--(parent)--(p1)--(p1p) (b)--(parent);
\draw (p1)--(ell1);
\draw (p2)--(ell2);

\node[below=10mm of b] () {(a)};

\begin{scope}[xshift=7cm]

\node[sdot,label={below:{$b$}}] (b) {}; 
\node[sdot,above=10mm of b,label={right:{$p_b$}}] (parent) {}; 
\node[sdot,above right=15mm of parent,label={right:{$g_b$}}] (p2) {}; 
\node[sdot,above left=15mm of parent,label={right:{$p_a$}}] (p1) {}; 
\node[above=6mm of p2] (p2p) {}; 
\node[sdot,above=15mm of p1,label={left:{$g_a$}}] (p1p) {}; 
\node[above=6mm of p1p] (p1pp) {};
\node[below right=10mm of p2,draw,rectangle,label={below:{$C_{g_b}$}}] (ell2) {
   \begin{tikzpicture}
      \node[below right=10mm of p2,draw,dashed,ellipse] (VGb) {$V_{g_b}$};
   \end{tikzpicture}
}; 

\node[sdot,below left=15mm of p1,label={below:{$a$}}] (ell1) {}; %
\node[below left=5mm of p1p] (p1pc) {};
\node[below left=10mm of p1pc,label={left:{$\ell$}}] (ell) {};
\draw (p2p)--(p2)--(parent)--(p1)--(p1p) (b)--(parent) (p1p)--(p1pc.center) (p1p)--(p1pp);
\draw[decorate,decoration={snake,amplitude=1mm,segment length=3mm}] (p1pc.center) to (ell.center);
\draw (p1)--(ell1);
\draw (p2)--(ell2);

\node[below=10mm of b] () {(b)};
\end{scope}
\begin{scope}[xshift=14cm]
\node[sdot,label={below:{$b$}}] (b) {}; 
\node[sdot,above=10mm of b,label={right:{$p_b$}}] (parent) {}; 
\node[sdot,above right=15mm of parent,label={right:{$g_b$}}] (p2) {}; 
\node[sdot,above left=15mm of parent,label={right:{$p_a$}}] (p1) {}; 
\node[above=6mm of p2] (p2p) {}; 
\node[sdot,above=30mm of parent,label={left:{$g_a$}}] (p1p) {}; 
\node[above=6mm of p1p] (p1pp) {};
\node[below right=10mm of p2,draw,rectangle,label={below:{$C_{g_b}$}}] (ell2) {
   \begin{tikzpicture}
      \node[below right=10mm of p2,draw,dashed,ellipse] (VGb) {$V_{g_b}$};
   \end{tikzpicture}
}; 
\node[sdot,below left=15mm of p1,label={below:{$a$}}] (ell1) {}; %
\node[below left=5mm of p1p] (p1pc) {};
\draw (p2)--(parent)--(p1) (b)--(parent) (p1p)--(p1pp);
\draw (p1)--(ell1);
\draw (p2)--(ell2);
\draw (p1) to [out=80,in=-135] (p1p);
\draw[dashed] (p2) to [out=100,in=-45] (p1p);

\node[below=10mm of b] () {(c)};
\end{scope}
\end{tikzpicture}
}
\caption{(a) shows the generic reticulated cherry scenario for Lemma~\ref{l:Vgb.properties}, while (b) and (c) illustrate scenarios for the proof of part (vi) of Lemma~\ref{l:Vgb.properties}. The wavy line in (b) indicates a tree path from $g_a$ to leaf $\ell$, while the dashed line in (c) indicates a (directed) path from $g_a$ to~$g_b$.}
\label{f:proof.of.5}
\end{figure}
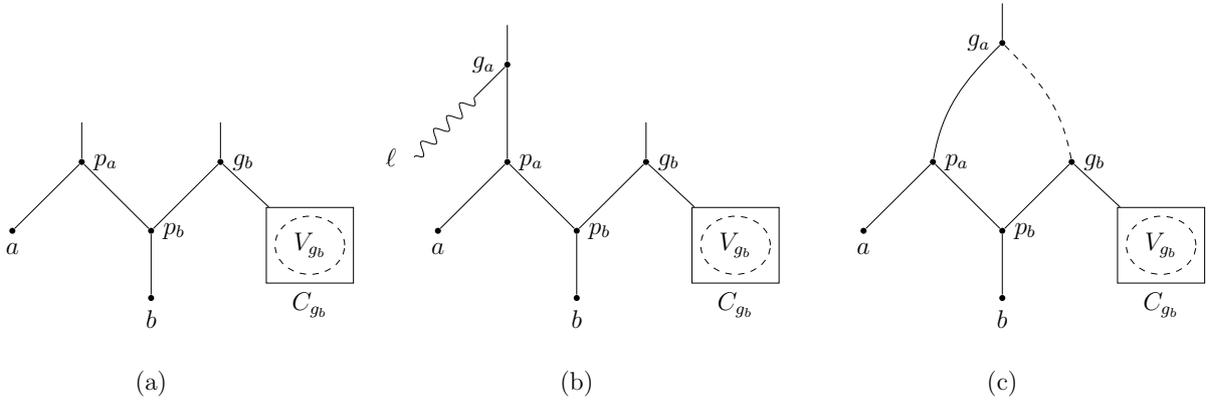

\noindent {\bf Remark.} The scenario in Lemma~\ref{l:Vgb.properties}(vi) also implies $cx|a, cx|b\not\in R(N)$ as an outcome, but this is not needed for our later arguments.

\section{Detecting Cherries and Reticulated Cherries From Rooted Triples}
\label{recognition}

A key part of proving Theorem~\ref{main} is to show that if $N$ is a normal network with no near reticulations, then cherries and reticulated cherries are recognisable using only the rooted triples of $N$. The next lemma \cite[Lemma~2]{linz2020caterpillars} shows that \emph{cherries} are recognisable from the set of rooted triples of an arbitrary normal network.

\begin{lem}
\label{l:cherry}
Let $N$ be a normal network on $X$ with $\{a, b\}\subseteq X$.  Then $\{a, b\}$ is a cherry of $N$ if and only if, whenever $xy|z\in R(N)$ and $a, b\in \{x, y, z\}$, we have $\{a, b\}=\{x, y\}$.
\end{lem}

For the proof of Theorem~\ref{main}, in addition to recognising reticulated cherries $\{a, b\}$ in a normal network with no near reticulations, where $b$ is the reticulation leaf, we need to recognise the visibility set of the grandparent $g_b$ of $b$ that is not the parent of $a$. The following definition captures sufficient properties of $V_{g_b}$ for this task.

\begin{defn}
\label{d:candidate.set} 
Let $N$ be a normal network with no near reticulations. A \emph{candidate set for $b$} is a non-empty subset $W_b$ of $X-\{a, b\}$ with the following properties:
\begin{enumerate}[{\rm (W1)}]
\item For all $c\in W_b$ and $x\in X-(W_b\cup \{b\})$, we have $bc|x\in R(N)$ and $ac|b\not\in R(N)$.

\item For distinct $c, c'\in W_b$, we have $bc|c'\not\in R(N)$.

\item For all $c\in W_b$ and $x\in X-(W_b\cup \{a, b\})$, we have $cx|a\in R(N)$ if and only if $bx|a\in R(N)$.

\item If $bx|y\in R(N)$ and $ax|y\not\in R(N)$, where $x, y\in X-(W_b\cup \{a, b\})$, then $cx|y\in R$ for all $c\in W_b$.

\item If $c\in W_b$ and $ac|x\in R(N)$ for all $x\in X-(W_b\cup \{a, b\})$, then $ax|c, ax|b\not\in R(N)$.
\end{enumerate}
\end{defn}

\noindent Note that these five properties for $W_b$ are also the properties (ii)--(vi) satisfied by the visibility set of $g_b$ in the context of Lemma~\ref{l:Vgb.properties}.

\begin{lem}
\label{l:retic.cherry}
Let $N$ be a normal network on $X$ with no near reticulations. Let $\{a, b\}\subseteq X$ and let $W_b$ be a candidate set for $b$. Then $\{a, b\}$ is a reticulated cherry of $N$ with reticulation leaf $b$, and $V_{g_b}=W_b$, where $g_b$ is the grandparent of $b$ that is not the parent of $a$, if and only if $ab|x\in R(N)$ for all $x\in X-\{a, b\}$.
\end{lem}

\begin{proof}
For the forward direction, if $\{a, b\}$ is a reticulated cherry, then, by Lemma~\ref{l:Vgb.properties}(i), $ab|x\in R(N)$ for all $x\in X-\{a, b\}$. For the backward direction, suppose that $ab|x\in R(N)$ for all $x\in X-\{a, b\}$. We first show that the parent $p_b$ of $b$ is a reticulation.

Suppose to the contrary that $p_b$ is a tree vertex. Let $v$ denote the child of $p_b$ that is not $b$.
Over the next several sub-lemmas, we deduce consequences of the assumption that $p_b$ is a tree vertex.  The first is that this implies that $v$ is a reticulation.

\begin{sublemma}
$v$ is a reticulation.
\label{reticulation}
\end{sublemma}

If $v$ is a leaf, then $v=a$; otherwise, $ab|v\not\in R(N)$, a contradiction. But if $v=a$, then, $\{a, b\}$ is a cherry and we cannot have $bc|a\in R(N)$, where $c\in W_b$, contradicting (W1) for $W_b$.  So $v$ is not a leaf.  

Suppose instead that $v$ is a tree vertex. Since $N$ is normal, there is a tree path from $v$ to a leaf, say $\ell$. If $a\neq \ell$, then either $a\ell|b\in R(N)$ or $b\ell|a\in R(N)$ but $ab|\ell\not\in R(N)$, a contradiction. So $a=\ell$. But then, similarly, $bc|a\not\in R(N)$, where $c\in W_b$, contradicting (W1) for $W_b$. Therefore $v$ is a reticulation and (\ref{reticulation}) holds.

Let $p_v$ denote the parent of $v$ that is not $p_b$, and let $\ell$ and $m$ be leaves at the end of tree paths starting at $v$ and $p_v$, respectively, as shown in Figure~\ref{f:scenarios.a.c.in.Cv}(a).  
Observe that $\ell\neq m$. Let $h_b$ denote the parent of $p_b$ in $N$. Since $N$ has no near-stack reticulations, $h_b$ is a tree vertex or the root of $N$. Furthermore, as $N$ has no near-sibling reticulations, the child of $h_b$ that is not $p_b$ is either a tree vertex or a leaf. Thus there is a tree path from $h_b$ to a leaf, say $z$, avoiding $p_b$. In particular, $z\neq b$. We next consider four cases depending on whether $a\in C_v$ and whether $W_b\cap C_v$ is non-empty.

\begin{figure}[ht]
\begin{tikzpicture}[sdot/.style={circle,fill,radius=1pt,inner sep=1pt}]
        \node[sdot,label={below:{$b$}}] (b) {}; 
        \node[sdot,above right=15mm of b,label={right:{$p_b$}}] (parent) {}; 
        \node[sdot,below right=15mm of parent,label={right:{$v$}}] (v) {}; 
        \node[sdot,above=10mm of parent,label={right:{$h_b$}}] (gp) {}; 
        \node[above=6mm of gp] (gpp) {}; 
        \node[below left=10mm of gp,label={left:{$z$}}] (z) {}; 
        \node[sdot,above right=10mm of v,label={right:{$p_v$}}] (vp) {}; 
        \node[above=6mm of vp] (vpp) {}; 
        \node[below right=10mm of vp,label={below:{$m$}}] (vpc) {}; 
        \node[below=10mm of v,label={left:{$\ell$}}] (ell) {}; 

        \draw[] (b) to (parent);
        \draw (gpp)--(gp)--(parent)--(v)--(vp); 
        \draw[decorate,decoration={snake,amplitude=1mm,segment length=3mm}] (v.center) to (ell);
        \draw (vpp)--(vp);
        \draw[decorate,decoration={snake,amplitude=1mm,segment length=3mm}] (vp.center) to (vpc.center);
        \draw[decorate,decoration={snake,amplitude=1mm,segment length=3mm}] (gp.center) to (z.center);

        \node[below=6mm of ell] () {(a)};

\begin{scope}[xshift=7cm]
\node[sdot,label={below:{$b$}}] (b) {};

\node[sdot,label={below:{$b$}}] (b) {}; 
\node[sdot,above right=15mm of b,label={right:{$p_b$}}] (parent) {}; 
\node[sdot,below right=15mm of parent,label={right:{$v$}}] (v) {}; 
\node[sdot,above=30mm of v,label={right:{$h_b$}}] (gp) {}; 
\node[above=6mm of gp] (gpp) {}; 
\node[sdot,above right=15mm of v,label={right:{$p_v$}}] (vp) {}; 
\node[below right=10mm of vp,label={below:{$m$}}] (vpc) {}; 
\node[below=10mm of v,label={below:{$\ell$}}] (ell) {}; 

\node[sdot,below right=5mm of gp,label={right:{$w$}}] (w) {};

\draw[] (b) to (parent);
\draw (gpp)--(gp) (parent)--(v)--(vp); 
\draw (gp) to[out=-150,in=85] (parent);
\draw[decorate,decoration={snake,amplitude=1mm,segment length=3mm}] (v.center) to (ell);
\draw[decorate,decoration={snake,amplitude=1mm,segment length=3mm}] (vp.center) to (vpc.center);
\draw (gp) to (w);
\draw[dashed] (w) to[out=-45,in=95] (vp.center);

\node[below=6mm of ell] () {(b)};
\end{scope}

\end{tikzpicture}  
\caption{Scenarios for (\ref{first}), (\ref{second}), (\ref{third}), and (\ref{sublem:pb.retic}), in which it is assumed that $p_b$ is a tree vertex.  The parent $h_b$ of $p_b$ must also be a tree vertex, otherwise $N$ has a near-stack reticulation. 
(a) There is a tree path from $h_b$ to a leaf $z$, and from $p_v$ to a leaf $m$. (b) The scenario in which $h_b$ is an ancestor of $p_v$, in which case there is a tree path from $h_b$ to $z$ traversing $w$.
}
\label{f:scenarios.a.c.in.Cv}
\label{f:hb.above.pv}
\end{figure}
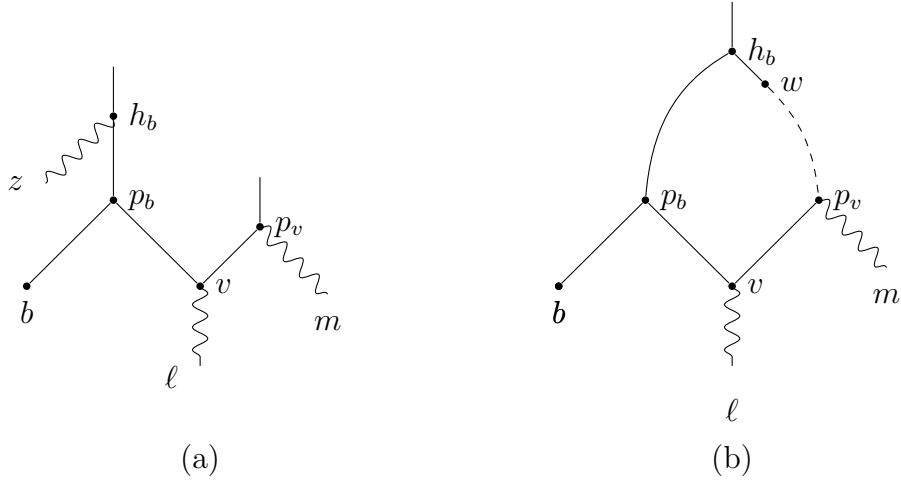

\begin{sublemma}
Either $a\not\in C_v$ or $W_b\cap C_v$ is empty.
\label{first}
\end{sublemma}

If $a\in C_v$ and there exists $c\in W_b\cap C_v$, then $ac|b\in R(N)$, contradicting (W1) for $W_b$. Thus (\ref{first}) holds.

\begin{sublemma}
Either $a\in C_v$ or $W_b\cap C_v$ is non-empty.
\label{second}
\end{sublemma}

Assume that $a\not\in C$ and $W_b\cap C_v$ is empty. Say $h_b$ is not an ancestor of $p_v$. Since $ab|x\in R(N)$ for all $x\in X-\{a, b\}$, it follows that $a\in C_{h_b}-(C_v\cup \{b\})$. In particular, $z=a$; otherwise, $ab|z\not\in R(N)$. But then, as $W_b\cap C_v$ is empty, either $ac|b\in R(N)$ or $ab|c\in R(N)$ but $bc|a\not\in R(N)$, contradicting (W1) for $W_b$. Now say $h_b$ is an ancestor of $p_v$, as in Figure~\ref{f:scenarios.a.c.in.Cv}(b), and note that $m$ and $z$ need not be distinct. Again, as $ab|x\in R(N)$ for all $x\in X-\{a, b\}$, we have $z=a$. But, as $W_b\cap C_v$ is empty, $bc|a\not\in R(N)$ for all $c\in W_b$, contradicting (W1) for $W_b$. Thus (\ref{second}) holds.

It now follows that either $a\not\in C_v$ and $W_b\cap C_v$ is non-empty, or $a\in C_v$ and $W_b\cap C_v$ is empty.

\begin{sublemma}
$a\in C_v$ and $W_b\cap C_v$ is empty.
\label{third}
\end{sublemma}

Assume that $a\not\in C_v$ and $W_b\cap C_v$ is non-empty. Note that $W_b\subseteq C_v$; otherwise, there is a $c'\in W_b$ such that $bc|c'\in R(N)$, contradicting (W2) for $W_b$. Say $h_b$ is not an ancestor of $p_v$. If $a\neq z$, then, as $a\not\in C_v$, it follows that  $ab|z\not\in R(N)$. So $a=z$. If $c\in W_b\cap C_v$, then $cm|a\in R(N)$, but $bm|a\not\in R(N)$, contradicting (W3) for $W_b$.

Now say $h_b$ is an ancestor of $p_v$ as in Figure~\ref{f:scenarios.a.c.in.Cv}(b) and, again, note that $m$ and $z$ may not be distinct.
Consider $a$, $b$, and $z$.  If $a\neq z$, then $ab|z\not\in R(N)$, a contradiction, and so $a=z$.  But then $ac|b\in R(N)$, where $c\in W_b$, contradicting (W1) for $W_b$. Hence~(\ref{third}) holds.

We now work towards a contradiction that will allow us to conclude that $p_b$ must be a reticulation. By (\ref{third}),  $a\in C_v$ and $W_b\cap C_v$ is empty. First assume that $h_b$ is not an ancestor of~$p_v$.
Consider $a$, $b$, and $z$. If $z\not\in W_b$, then, for all $c\in W_b$, we have $bc|z\not\in R(N)$ as $W_b\cap C_v$ is empty, contradicting (W1) for $W_b$. So $z\in W_b$.  Also, if there is a $t\in V_{h_b}-(W_b\cup \{b\})$, then $bc|t\not\in R(N)$ for all $c\in W_b$, a contradiction.  So $V_{h_b}\subseteq W_b\cup \{b\}$.

We proceed by showing that $ac|x\in R(N)$ for all $x\in X-(W_b\cup \{a, b\})$, where $c=z$, and then use (W5) for $W_b$ to get a contradiction. Let $y\in X-(W_b\cup \{a, b\})$. By assumption, $ab|y\in R(N)$, and so there is an embedding $E_{aby}$ of $ab|y$ in $N$. Label the arcs of the rooted triple $ab|y$ by $f_a$, $f_b$, $f_{ab}$, and $f_y$, where $f_a$, $f_b$, and $f_y$ are the pendant edges incident with $a$, $b$ and $y$, respectively (along the lines of the arc-labelling in Figure~\ref{f:edge.labels}). If $E_{aby}$ uses the arc $(p_b,v)$, and so $f_{ab}$ uses $(h_b, p_b)$ in $E_{aby}$, then it is straightforward to modify $E_{aby}$ to get an embedding of $ac|y$ in $N$ unless $f_y$ meets $f_{ab}$ at $h_b$. In this case, as $y\not\in V_{h_b}$, there is a path $P_y$ from the root of $N$ to $y$ avoiding $h_b$. Let $P_c$ denote the tree path from $h_b$ to $c$, and let $P_a$ denote the path that is the union of $\{(h_b, p_b)\}$ and the subset of arcs of $E_{aby}$ corresponding to $f_a$. If $P_y$ and $P_c$ meet, then $P_c$ contains a reticulation, a contradiction. Hence if $P_{h_b}$ is a path from the root of $N$ to $h_b$, it is now easily seen that
$$E(P_a)\cup E(P_c)\cup E(P_{h_b})\cup E(P_y)$$
contains an embedding of $ac|y$ in $N$, and so $ac|y\in R(N)$. If $E_{aby}$ does not use the arc $(p_b, v)$, then $f_b$ uses $h_b$ in $E_{aby}$. So replace the arcs corresponding to $f_a$ in $E_{aby}$ with the arcs in the tree path from $h_b$ to $c$ to get an embedding of $ac|y$ in $N$. It now follows that $ac|x\in R(N)$ for all $x\in X-(W_b\cup \{a, b\})$, contradicting (W5) for $W_b$.

Now assume that $h_b$ \emph{is} an ancestor of $p_v$, as in Figure~\ref{f:scenarios.a.c.in.Cv}(b).
If $z\neq c$, then, as $W_b\cap C_v$ is empty, $bc|z\not\in R(N)$, a contradiction to (W2) for $W_b$.  So $z=c$ and, as $a\in C_v$, we have $ac|b\in R(N)$, contradicting (W1) for $W_b$.

We conclude that $p_b$ is not a tree vertex, and so we have (\ref{sublem:pb.retic}):
\begin{sublemma}
\label{sublem:pb.retic}
$p_b$ is a reticulation.
\end{sublemma}

Let $p_1$ and $p_2$ denote the parents of $p_b$, and let $\ell_1$ and $\ell_2$ denote the leaves at the end of tree paths starting at $p_1$ and $p_2$, respectively, as shown in Figure~\ref{f:wheres.a}(a). To complete the proof of the lemma, it remains to show that, for $\{i, j\}=\{1, 2\}$, the vertex $p_i$ is the parent of $a$ and $V_{p_j}=W_b$.

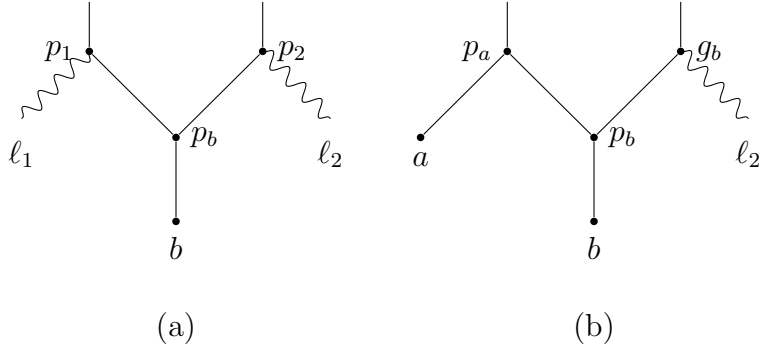
\begin{figure}[ht]
\begin{tikzpicture}[sdot/.style={circle,fill,radius=1pt,inner sep=1pt}]
\node[sdot,label={below:{$b$}}] (b) {}; 
\node[sdot,above=10mm of b,label={right:{$p_b$}}] (parent) {}; 
\node[sdot,above right=15mm of parent,label={right:{$p_2$}}] (p2) {}; 
\node[sdot,above left=15mm of parent,label={left:{$p_1$}}] (p1) {}; 
\node[above=6mm of p2] (p2p) {}; 
\node[above=6mm of p1] (p1p) {}; 
\node[below right=10mm of p2,label={below:{$\ell_2$}}] (ell2) {}; %
\node[below left=10mm of p1,label={below:{$\ell_1$}}] (ell1) {};
        
\draw (p2p)--(p2)--(parent)--(p1)--(p1p) (b)--(parent);
\draw[decorate,decoration={snake,amplitude=1mm,segment length=3mm}] (p2.center) to (ell2.center);
\draw[decorate,decoration={snake,amplitude=1mm,segment length=3mm}] (p1.center) to (ell1.center);
\node[below=10mm of b] () {(a)};

\begin{scope}[xshift=55mm]
\node[sdot,label={below:{$b$}}] (b) {}; 
\node[sdot,above=10mm of b,label={right:{$p_b$}}] (parent) {}; 
\node[sdot,above right=15mm of parent,label={right:{$g_b$}}] (p2) {}; 
\node[sdot,above left=15mm of parent,label={left:{$p_a$}}] (p1) {}; 
\node[above=6mm of p2] (p2p) {}; 
\node[above=6mm of p1] (p1p) {}; 
\node[below right=10mm of p2,label={below:{$\ell_2$}}] (ell2) {}; %
\node[sdot,below left=15mm of p1,label={below:{$a$}}] (ell1) {}; %
        
\draw (p2p)--(p2)--(parent)--(p1)--(p1p) (b)--(parent);
\draw[decorate,decoration={snake,amplitude=1mm,segment length=3mm}] (p2.center) to (ell2.center);
\draw (p1)--(ell1);
\node[below=10mm of b] () {(b)};

\end{scope}

\end{tikzpicture}
\caption{Illustration of scenarios for the argument after the parent $p_b$ of $b$ has been identified as a reticulation.  (a) Once $b$ is identified as a reticulation, we need to be able to use the rooted triples to show that for some $i\in \{1, 2\}$, the vertex $p_i$ is the parent of $a$ and (b) to show that, for the other parent of $p_b$, its visibility set is $W_b$.}
\label{f:wheres.a}
\end{figure}

\begin{sublemma}
If $W_b\cap V_{p_i}$ is non-empty for some $i\in \{1, 2\}$, then $W_b\subseteq V_{p_i}$.
\label{contained}
\end{sublemma}

Say $W_b\cap V_{p_i}$ is non-empty for some $i\in \{1, 2\}$. If there is an element $c'\in W_p-V_{p_i}$, then $bc|c'\in R(N)$, where $c\in W_b\cap V_{p_i}$, contradicting (W2). So (\ref{contained}) holds.

\begin{sublemma}
If $a\in C_{p_i}$ for some $i\in \{1, 2\}$, then $C_{p_i}=\{a, b\}$ and $W_b=V_{p_j}$, where $\{i, j\}=\{1, 2\}$.
\label{cluster1}
\end{sublemma}

Say $a\in C_{p_i}$. If $\ell_j\not\in W_b$, then, as $bc|\ell_j\in R(N)$, an embedding $E_{bc\ell_j}$ of $bc|\ell_j$ in $N$ must use $(p_i, p_b)$, where $c\in W_b$. But then we can modify $E_{bc\ell_j}$ to get an embedding of $ac|\ell_j$ that uses $p_i$ which in turn can be modified to get an embedding of $ac|b$ in $N$, a contradiction to (W1) for $W_b$. Thus $\ell_j\in W_b$. In particular, $V_{p_j}\subseteq W_b$, and so, by (\ref{contained}), $V_{p_j}=W_b$. 

Now consider $C_{p_i}$. If $a\not\in V_{p_i}$, then $a\neq \ell_i$ and an embedding $E_{ab\ell_i}$ of $ab|\ell_i$ in $N$ uses $(p_j, p_b)$. But then $ac|\ell_i\in R(N)$ and, in turn, $ac|b\in R(N)$, where $c\in W_b$, a contradiction. Thus $a\in V_{p_i}$. If $|V_{p_i}|\ge 2$, then, for some element $z\in V_{p_i}-\{a\}$, we have $ab|z\not\in R(N)$, a contradiction. So $|V_{p_i}|=1$ and $V_{p_i}=\{a\}$. If $|C_{p_i}|\ge 3$, then the child of $p_i$ that is not $p_b$, say $q_i$, is a tree vertex as $N$ has no sibling reticulations. If there is a tree path from $q_i$ to a leaf other than $a$, then $|V_{p_i}|\ge 2$, a contradiction. So $q_i$ is the parent of a reticulation, and this reticulation and $p_b$ are near-sibling reticulations, a contradiction. Hence $C_{p_i}=\{a, b\}$, and so (\ref{cluster1}) holds.

\begin{sublemma}
If $W_b\cap C_{p_i}$ is non-empty for some $i\in\{1, 2\}$, then $W_b=V_{p_i}$ and $C_{p_j}=\{a, b\}$, where $\{i, j\}=\{1, 2\}$.
\label{cluster2}
\end{sublemma}

Say $W_b\cap C_{p_i}$ is non-empty. If $a\neq \ell_j$, then an embedding $E_{ab\ell_j}$ of $ab|\ell_j$ in $N$ uses $(p_i, p_b)$. But then we can modify $E_{ab\ell_j}$ to get an embedding of $ac|\ell_j$ that uses $p_i$, which in turn gives an embedding of $ac|b$ in $N$, a contradiction to (W1) for $W_b$. So $a=\ell_j$. If $|V_{p_j}|\ge 2$, then there is an element $z\in V_{p_j}-\{a\}$ and $ab|z\not\in R(N)$, a contradiction. So $|V_{p_j}|=1$. If $C_{p_j}-\{a, b\}$ is non-empty, then the child of $p_j$ that is not $p_b$ is a tree vertex and, as $N$ has no near-sibling reticulations, it follows that there is a tree path from $p_j$ to a leaf that is not $a$, contradicting $|V_{p_j}|=1$. Hence $C_{p_j}=\{a, b\}$. By~(\ref{cluster1}), $W_b=V_{p_i}$, and (\ref{cluster2}) holds.

We complete the proof of the lemma by showing that, for some $i\in\{1, 2\}$, either $a\in C_{p_i}$ or $W_b\cap C_{p_i}$ is non-empty and then apply either (\ref{cluster1}) or (\ref{cluster2}). First assume that $p_1$ and $p_2$ have the same parent. If $a\not\in C_{p_1}\cup C_{p_2}$, then $ab|\ell_1\not\in R(N)$, a contradiction. Thus either $a\in C_{p_1}$ or $a\in C_{p_2}$, and the lemma follows by (\ref{cluster1}).

We may now assume that $p_1$ and $p_2$ do not have the same parent. Let $h_1$ and $h_2$ denote the parents of $p_1$ and $p_2$. Since $N$ has no near-stack reticulations, $h_1$ and $h_2$ are both tree vertices and so, as $N$ is acyclic, either $h_1$ is not an ancestor of $p_2$ or $h_2$ is not an ancestor of $p_1$. Without loss of generality, we may assume that $h_1$ is not an ancestor of $p_2$. Since $N$ has no near-sibling reticulations, the child of $h_1$ that is not $p_1$ is a tree vertex or a leaf, and so there is a tree path starting at $h_1$, avoiding $p_1$, and ending at a leaf $m_1\neq \ell_1$.

If $m_1\in W_b\cup \{a\}$, then either $ab|\ell_1\not\in R(N)$ or $bc|\ell_1\not\in R(N)$ for some $c\in W_b$, a contradiction. So $m_1\not\in W_b\cup \{a\}$. If $\ell_1\not\in W_b\cup \{a\}$, then, as $b\ell_1|m_1\in R(N)$ and $W_b$ is a candidate set for $b$, it follows by (W4) for $W_b$ that either $a\ell_1|m_1\in R(N)$ or $c\ell_1|m_1\in R(N)$ for all $c\in W_b$. This implies that $(W_b\cup \{a\})\cap C_{p_1}$ is non-empty and the lemma follows by (\ref{cluster1}) and (\ref{cluster2}). This completes the proof of Lemma~\ref{l:retic.cherry}.
\end{proof}

\section{Reconstructing Normal Networks From Rooted Triples}
\label{proof}

In this section, we prove Theorem~\ref{main}. We begin with a lemma that allows us to use induction for its proof.

\begin{lem}
Let $N$ be a normal network on $X$ with no near reticulations, and let $\{a, b\}$ be a reticulated cherry of $N$ with reticulation leaf $b$.  Let $N'$ be the phylogenetic network obtained from $N$ by deleting $b$ and its parent, and suppressing the resulting two vertices of degree~$2$. Then $N'$ is a normal network on $X-\{b\}$ with no near reticulations.
\label{still-normal}
\end{lem}

\begin{proof}
It follows by~\cite[Lemma~3.2]{bordewich2018recovering} that $N'$ is normal. To see that $N'$ has no near reticulations, suppose to the contrary that $u$ and $v$  are either near-sibling or near-stack reticulations of $N'$. Let $p_a$ and $p_b$ denote the parents of $a$ and $b$, respectively, in $N$ and let $g_b$ denote the parent of $p_b$ that is not $p_a$. Say $u$ and $v$ are near-sibling reticulations of $N'$, where the parent $p_v$ of $v$ is the parent of a parent $p_u$ of $u$. Then, as $N$ has no near reticulations, to obtain $N$ from $N'$ it follows that $g_b$ subdivides one $(p_v, p_u)$, $(p_u, u)$, and $(p_v, v)$ in $N$. If $g_b$ subdivides either $(p_u, u)$ or $(p_v, v)$, then either $u$ and $p_b$ or $v$ and $p_b$ are sibling reticulations in $N$, a contradiction. If $g_b$ subdivides $(p_v, p_u)$, then $v$ and $p_b$ are near-sibling reticulations in $N$, another contradiction. So $u$ and $v$ are not near-sibling reticulations. Now say $u$ and $v$ are near-stack reticulations of $N'$, where the child of $u$ is a parent $p_v$ of $v$. Since $N$ has no near reticulations, to obtain $N$ from $N'$, the vertex $g_b$ subdivides either $(u, p_v)$ or $(p_v, v)$ in $N$. But then either $u$ and $p_b$ are near-stack reticulations or $v$ and $p_b$ are sibling reticulations, respectively, in $N$. This last contradiction implies that $N'$ has no near reticulations, thereby completing the proof of the lemma.
\end{proof}

We now prove Theorem~\ref{main}.

\begin{proof}[Proof of Theorem~\ref{main}]
Let $N_1$ and $N_2$ be two normal networks on $X$ with no near reticulations. If $N_1\cong N_2$, then $R(N_1)=R(N_2)$. The proof of the converse is by induction on $|X|$. 

Suppose that $R(N_1)=R(N_2)$. If $|X|\in \{1, 2\}$, then neither $N_1$ nor $N_2$ has any reticulations and so $N_1\cong N_2$, and the theorem holds. Now suppose that $|X|\ge 3$ and that the theorem holds for all normal networks with no near reticulations on $|X|-1$ leaves. By Lemma~\ref{cherries}, $N_1$ has either a cherry or a reticulated cherry. If $N_1$ has a cherry $\{a, b\}$, then, by Lemma~\ref{l:cherry}, $\{a, b\}$ is recognised by $R(N_1)$. Since $R(N_1)=R(N_2)$, Lemma~\ref{l:cherry} also implies that $\{a, b\}$ is a cherry of $N_2$. Let $N'_1$ and $N'_2$ denote the normal networks obtained from $N_1$ and $N_2$, respectively, by deleting $b$ and suppressing the resulting degree-two vertex. 
Since neither $N_1$ nor $N_2$ has a near reticulation, we have from Lemma~\ref{still-normal} that neither $N'_1$ nor $N'_2$ has a near reticulation. 
Furthermore, as $R(N_1)=R(N_2)$, it follows that $R(N'_1)=R(N'_2)$, in particular, for each $i\in \{1, 2\}$,
$$R(N'_i) = R(N_i)-\{bx|y: x, y\in X-\{b\}\}.$$
Therefore, by induction, $N'_1\cong N'_2$. For each $i\in \{1, 2\}$, as $\{a, b\}$ is a cherry of $N_i$, the only way $b$ can be adjoined to $N'_i$ to get $N_i$ is to subdivide the arc directed into $a$ and adjoin $b$ to this new vertex with a new arc. Thus $N_1\cong N_2$.

Now assume that $N_1$ has a reticulated cherry $\{a, b\}$ with reticulation leaf $b$. By Lemma~\ref{l:retic.cherry}, this reticulated cherry, as well as the visibility set $V_{g_b}$ of the grandparent $g_b$ of $b$ that is not the parent of $a$, is recognised by $R(N_1)$. Since $R(N_1)=R(N_2)$, it follows that $\{a, b\}$ is also a reticulated cherry of $N_2$ with reticulation leaf $b$ and the visibility set of the grandparent of $b$ that is not the parent of $a$ is $V_{g_b}$. Let $N'_1$ and $N'_2$ denote the phylogenetic networks obtained from $N_1$ and $N_2$, respectively, by deleting $b$ and its parent, and suppressing the two resulting degree-two vertices. By Lemma~\ref{still-normal}, $N'_1$ and $N'_2$ are normal networks with no near reticulations. For each $i\in \{1, 2\}$, as the rooted triples displayed by $N'_i$ are exactly the rooted triples displayed by $N_i$ not having $b$ as a leaf, $R(N'_1)=R(N'_2)$. Therefore, by induction, $N'_1\cong N'_2$.

Let $p_a$ and $p_b$ denote the parents of $a$ and $b$, respectively, in $N_1$, and recall that $g_b$ is the parent of $p_b$ that is not $p_a$ in $N_1$. Let $e_1=(p_a, p_b)$ and $f_1=(g_b, p_b)$ denote the reticulation arcs in $N_1$ directed into $p_b$. Since $N'_1\cong N'_2$, to show that $N_1\cong N_2$, it suffices to show that there is exactly one arc in $N'_1$ to adjoin $e_1$ and exactly one arc in $N'_1$ to adjoin $f_1$. Since, up to suppressing degree-two vertices, $\{a, b\}$ is a cherry of the normal network obtained from $N_1$ by deleting $f_1$, the only arc to adjoin $e_1$ to in $N'_1$ is the arc directed into $a$.

We now consider the possible arcs of $N'_1$ to adjoin $f_1$ to. Since $g_b$ is a parent of the reticulation $p_b$ in $N_1$, it follows that the child of $g_b$ that is not $p_b$ is a tree vertex or a leaf, and so it has visibility set $V_{g_b}$. Thus there is at least one vertex in $N'_1$ with visibility set $V_{g_b}$. Let $U$ denote the subset of vertices of $N'_1$ with visibility sets $V_{g_b}$. Let $u\in U$. Since $N'_1$ is acyclic, we may assume that $u$ is chosen so that no ancestor of $u$ is in $U$.

\begin{sublemma}
$|U|\le 2$ and, if $U=\{u, u'\}$ and $u$ is a tree vertex, then $u'$ is a tree vertex child of $u$ and the other child of $u$ is a reticulation.
\label{most2}
\end{sublemma}

If $|U|=1$, then (\ref{most2}) is proved. Assume that $|U|\ge 2$, and let $u'\in U-\{u\}$. If $u'$ is not a descendant of $u$, then, by the choice of $u$, for each leaf $\ell\in V_u$, there is a path in $N'_1$ from the root to $\ell$ traversing $u$ but not $u'$, a contradiction. Thus $u'$ is a descendant of $u$. Let $P=u, v_1, v_2, \ldots, v_k, u'$ be a (directed) path from $u$ to $u'$ in $N'_1$.
If there is a tree path from $u$ to a leaf $m$ that avoids traversing $v_1$, then, by the choice of $u$, we have $m\in V_u$ but $m\not\in V_{u'}$, a contradiction. Therefore if $u$ is a tree vertex, it is the parent of a reticulation. In turn, this implies that, regardless of whether $u$ is a tree vertex or a reticulation, $v_1$ is not a reticulation as $N'_1$ is normal. So $v_1$ is a tree vertex. Let $w_1$ be the child of $v_1$ that is not $v_2$. If $w_1$ is a tree vertex, then there is a tree path from $v_1$ to a leaf $m_1$ that avoids traversing $v_2$, and so $m_1\in V_u$ but $m_1\not\in V_{u'}$, a contradiction. Thus $w_1$ is a reticulation. If $u$ is a reticulation, $u$ and $w_1$ are near-stack reticulations, while if $u$ is a tree vertex, $w_1$ and the child of $u$ that is not $v_1$ are near-sibling reticulations, a contradiction. Hence $v_1=u'$. 

It now follows that if $u$ is a reticulation, then $|U|=2$. Say $u$ is a tree vertex. Let $v$ be the child of $u$ that is not $u'$ and let $m$ be a leaf at the end of a tree path starting at $v$. If $v$ is a tree vertex, then $m\in V_u$ but $m\not\in V_{u'}$, a contradiction. So $v$ is a reticulation. As $N'_1$ has no shortcuts, $m\not\in V_u$ but $m$ is an element of the visibility set of $v$. Thus $v\not\in U$, so $|U|=2$ and (\ref{most2}) holds.

If $u$ is a reticulation, then, by (\ref{most2}), $U=\{u, u'\}$ and $f_1$ adjoins to either one of the arcs directed into $u$, in which case $u$ and $p_b$ are sibling reticulations of $N_1$, or the arc $(u, u')$, in which case $u$ and $p_b$ are near-stack reticulations of $N_1$, a contradiction. Thus $u$ is not a reticulation, and so $u$ is a tree vertex. If $|U|=2$, then, by (\ref{most2}), $f_1$ adjoins to either the arc directed into $u$ or the arc $(u, u')$. In both instances, $p_b$ and the reticulation child of $u$ are near-sibling reticulations of $N_1$, another contradiction. Hence $|U|=1$, and so there is exactly one arc to adjoin $f_1$, namely, the arc directed into $u$. We conclude that $N_1\cong N_2$, thereby completing the proof of Theorem~\ref{main}.
\end{proof}

\section{Discussion}
\label{discussion}

The capability to reliably reconstruct an explicit phylogenetic network from a set of genetic sequences is a key missing element preventing their widespread use in biological applications. Acquiring this capability involves two logical steps.  

The first step is obtaining results that show phylogenetic networks \emph{can} be reconstructed given particular information (and preferably giving a procedure for doing so).  The second step is being able to tell when one has the right information in the first place.  For instance, the present paper has shown that a certain class of phylogenetic networks can be reconstructed from the sets of rooted triples that they display.  But it hasn't give a straightforward method to determine whether a given set of rooted triples might be displayed by a network of that class.  

Another way to view these two steps is in terms of mathematical maps.  In showing that a class $\mathcal N$ of phylogenetic networks can be reconstructed from the rooted triples it displays, we are effectively showing that the map from $\mathcal N$ to the power set of rooted triples is injective.  In doing that, we also show how to reconstruct the pre-image of the set of rooted triples.  The second step would be to describe the image of this map, so that we have a bijection between the class of normal networks without near reticulations, and a well-described subset of the power set of the rooted triples.

The second step is crucial for applications: a scientist who gathers a set of genetic sequences does not \emph{a priori} know whether their data evolved on a normal network without near reticulations.  Ideally, they would be able to compute the set of rooted triples represented by their data, and be able to tell from those rooted triples whether they have the capability to reconstruct.  With the results in this paper, the test for whether a given set of rooted triples is the (complete) set of rooted triples displayed by a normal network, $N$ say, with no near reticulations, would involve finding a cherry or a set $W_b$ among the rooted triples, performing a cherry or reticulated cherry reduction, and repeating until it were no longer possible.  A more practical test, and one not requiring the complete set of rooted triples displayed by $N$, would be better.

In this paper, we have shown that normal networks without near reticulations can be reconstructed from the set of rooted triples they display. Thus, excluding near-sibling and near-stack reticulations is \emph{sufficient} for reconstruction, but are these exclusions also \emph{necessary} for reconstruction? As illustrated in \cite[Figure~1]{linz2020caterpillars}, it is not possible for \emph{all} normal networks to be reconstructed from their rooted triples. {In particular, Lemma~\ref{l:near-sib.retics.same.triples} shows that, if $N$ is a normal network with a pair of near-sibling reticulations $u$ and $v$, and $u$ and $v$ are non-comparable, then $N$} cannot be distinguished by rooted triples alone. Hence, excluding {non-comparable} near-sibling reticulations is \emph{necessary} for reconstruction. {This leaves open two questions: (i) To what extent is excluding comparable near-sibling reticulations necessary for reconstructing normal networks and (ii) is excluding near-stack reticulations necessary for reconstruction? For (ii), if $N$ is a normal network and $N$ has no near-sibling reticulations, we have Conjecture~\ref{c:near.sib.enough}. However, if $N$ is allowed to have comparable near-sibling reticulations, then the conjecture is false as the following example highlights.}

{Consider the normal network $N$ shown in Figure~\ref{f:intertwined}(a) and observe that $u$ and $v$ are near-sibling as well as near-stack reticulations. Using an approach similar to that to prove Lemma~\ref{l:near-sib.retics.same.triples}, it is straightforward to show that $R(N)=R(N')$, where $N'$ is the normal network shown in Figure~\ref{f:intertwined}(b), and $N$ is not isomorphic to $N'$. Additionally, this example highlights the potential difficulties in answering (i). What if the arc from $v$ to $q_u$ in $N$ is a directed path, and so $u$ and $v$ are no longer near-stack reticulations? Initial investigations show that the answer to (i) is dependent on how the rest of $N$ interacts with this directed path, and it is entirely possible that the answer to (i) is not succinct.}

\begin{figure}[ht]
\begin{tikzpicture}[sdot/.style={circle,fill,radius=1pt,inner sep=1pt}]
\node[sdot,label={left:{$u$}}] (u) {}; 
\node[sdot,above left=10mm of u,label={left:{$q_u$}}] (qu) {}; 
\node[sdot,above right=10mm of u,label={right:{$p_u$}}] (pu) {}; 
\node[sdot,above=10mm of qu,label={right:{$v$}}] (v) {}; 
\node[sdot,above left=10mm of v,label={left:{$q_v$}}] (qv) {}; 
\node[sdot,above right=10mm of v,label={right:{$q_v$}}] (pv) {}; 
\node[above=6mm of qv] (qvp) {}; 
\node[above=6mm of pv] (pvp) {}; 
\node[sdot,below left=6mm of qv, label={right:{$w_1$}}] (qvc) {}; 
\node[left= 6mm of qvc] (qvc1) {}; 
\node[below=6mm of qvc] (qvc2) {}; 
\node[sdot,below left=6mm of qu, label={right:{$w_2$}}] (quc) {}; 
\node[left= 6mm of quc] (quc1) {}; 
\node[below=6mm of quc] (quc2) {}; 
\node[sdot,below right=6mm of pu, label={left:{$w_4$}}] (puc) {}; 
\node[below= 6mm of puc] (puc1) {}; 
\node[right= 6mm of puc] (puc2) {}; 
\node[sdot,below=6mm of u, label={left:{$w_3$}}] (uc) {}; 
\node[below left= 6mm of uc] (uc1) {}; 
\node[below right= 6mm of uc] (uc2) {}; 

\draw (uc1)--(uc)--(uc2)
      (puc1)--(puc)--(puc2)
      (quc1)--(quc)--(quc2)
      (qvc1)--(qvc)--(qvc2)
      (quc)--(qu)--(u)--(uc)
      (u)--(pu)--(puc)
      (pu)--(pv)--(pvp)
      (pv)--(v)--(qv)--(qvp)
      (qv)--(qvc);
\draw (qu)--(v);

\node[below=15mm of u] () {(a) $N$};

\begin{scope}[xshift=65mm]
\node[sdot,label={left:{$u'$}}] (u') {}; 
\node[sdot,above left=10mm of u',label={left:{$q_v$}}] (qv) {}; 
\node[sdot,above right=10mm of u',label={right:{$p_{u'}$}}] (pu') {}; 
\node[sdot,below left=6mm of qv, label={right:{$w_1$}}] (qvc) {}; 
\node[left= 6mm of qvc] (qvc1) {}; 
\node[below=6mm of qvc] (qvc2) {}; 
\node[sdot,below right=6mm of pu', label={left:{$w_3$}}] (pu'c) {}; 
\node[below= 6mm of pu'c] (pu'c1) {}; 
\node[right= 6mm of pu'c] (pu'c2) {}; 
\node[sdot,below=6mm of u', label={right:{$w_2$}}] (u'c) {}; 
\node[below left= 6mm of  u'c] (u'c1) {}; 
\node[below right= 6mm of u'c] (u'c2) {}; 

\node[sdot,above=10mm of pu',label={right:{$v'$}}] (v') {}; 
\node[sdot,above left=10mm of v',label={left:{$q_{v'}$}}] (qv') {}; 
\node[sdot,above right=10mm of v',label={right:{$p_u$}}] (pu) {}; 
\node[above=6mm of qv'] (qv'p) {}; 
\node[above=6mm of pu] (pup) {}; 
\node[sdot,below right=6mm of pu, label={left:{$w_4$}}] (puc) {}; 
\node[below= 6mm of puc] (puc1) {}; 
\node[right=6mm of puc] (puc2) {}; 

\draw (u'c1)--(u'c)--(u'c2)
      (pu'c1)--(pu'c)--(pu'c2)
      (puc1)--(puc)--(puc2)
      (qvc1)--(qvc)--(qvc2)
      (qvc)--(qv)
      (u')--(u'c)
      (u')--(pu')--(pu'c)
      (pu')--(v')--(pu)--(puc)
      (pup)--(pu)--(v')--(qv')--(qv'p)
      (pu)--(puc)
      (qv)--(qv');

\draw (qv)--(u');

\node[below=15mm of u'] () {(b) $N'$};
\end{scope}
\end{tikzpicture}
\caption{
(a) A normal network $N$, where the reticulations $u$ and $v$ are both near-sibling and near-stack reticulations. (b) A normal network with $R(N)=R(N')$.
}
\label{f:intertwined}
\end{figure}
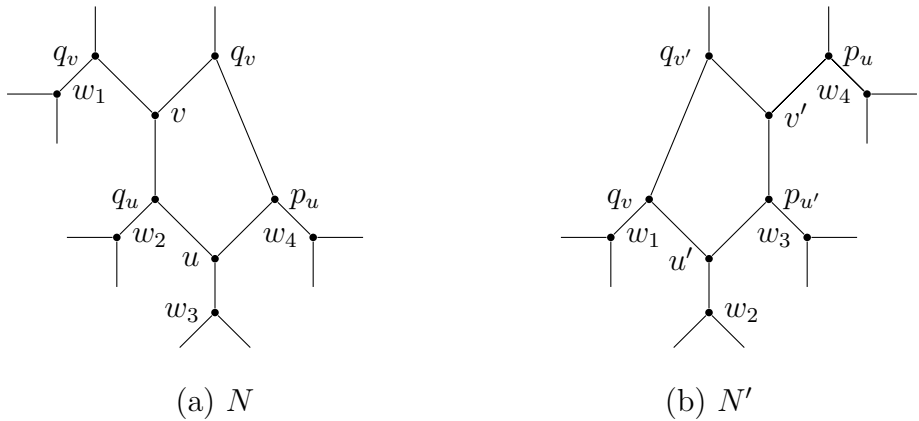

\section*{Data Availability Statement}
This manuscript has no associated data.

\bibliographystyle{plain}

\end{document}